\documentclass[11pt, a4paper, twoside,reqno]{amsart}

\usepackage[utf8]{inputenc}
\usepackage[british]{babel}
\usepackage[T1]{fontenc}
\usepackage{amsmath,amsfonts,amsthm,latexsym,amssymb}
\usepackage{microtype}
\usepackage{fancyhdr} 
\usepackage{lmodern}
\usepackage{xcolor}
\usepackage[a4paper]{geometry}

\usepackage{tikz-cd}
\usepackage{nccmath}
\usepackage{enumerate}
\usepackage{empheq}
\usepackage[bookmarks=true,breaklinks=true,
bookmarksnumbered = true,colorlinks= true,urlcolor= blue,citecolor=blue]{hyperref}
\usepackage{theoremref}
\usepackage{float}

\usepackage{faktor}
\allowdisplaybreaks

\usepackage{url}
\usepackage[noadjust]{cite}

\usepackage{graphicx}
\usepackage{pdflscape}
\usepackage{epstopdf}

\theoremstyle{plain}
\newtheorem{thm}{Theorem}[section]
\newtheorem{prop}[thm]{Proposition}

\newtheorem{lem}[thm]{Lemma}
\newtheorem*{acknow}{Acknowledgements}
\theoremstyle{remark}
\newtheorem{rem}[thm]{Remark}

\theoremstyle{definition}
\newtheorem{defin}[thm]{Definition}

\newcommand\N{\mathbb{N}}
\newcommand\Z{\mathbb{Z}}

\newcommand\WB[1]{WB_{#1}}	

\newcommand\ii{i}
\newcommand\inv{^{-1}}

\newcommand\jj{j}

\newcommand\nn{n}
\newcommand\nno{{n-1}}

\newcommand\rr[1]{\rho_{#1}^{\phantom{1}}}	
\newcommand\rrq[2]{\rho_{#1}^{#2}}	

\newcommand\sig[1]{\sigma_{#1}^{\phantom{1}}} 
\newcommand\sigg[2]{\sigma_{#1}^{#2}}	
\newcommand\siginv[1]{\sigg{#1}{-1}}	

\renewcommand{\ker}{\ensuremath{\operatorname{\text{Ker}}}}

\makeatletter

\renewcommand{\p@enumii}{}

\def\@enum@{\list{\csname label\@enumctr\endcsname}%
          {\usecounter{\@enumctr}\def\makelabel##1{
\normalfont\ignorespaces\emph{{##1}~}}
\setlength{\labelsep}{3pt}
\setlength{\parsep}{0pt}
\setlength{\itemsep}{0pt}
\setlength{\leftmargin}{0pt}
\setlength{\labelwidth}{0pt}
\setlength{\listparindent}{\parindent}
\setlength{\itemsep}{0pt}
\setlength{\itemindent}{0pt}
\topsep=3pt plus 1pt minus 1 pt}}
\makeatother

\renewcommand{\to}{\ensuremath{\longrightarrow}}

\begin{document}

\title[Unrestricted virtual braids and crystallographic braid groups]{Unrestricted virtual braids and crystallographic braid groups}

\author[Paolo Bellingeri]{Paolo Bellingeri}
\address{Normandie Univ, UNICAEN, CNRS, LMNO, 14000 Caen, France}
\email{paolo.bellingeri@unicaen.fr}

\author[John Guaschi]{John Guaschi}
\address{Normandie Univ, UNICAEN, CNRS, LMNO, 14000 Caen, France}
\email{john.guaschi@unicaen.fr}

\author[Stavroula Makri]{Stavroula Makri}
\address{Normandie Univ, UNICAEN, CNRS, LMNO, 14000 Caen, France}
\email{stavroula.makri@unicaen.fr}

\date{\today}





\begin{abstract}
We show that the crystallographic braid group $B_n/[P_n,P_n]$ embeds naturally in the group of unrestricted virtual braids $UVB_n$, we give new proofs of known results about the torsion elements of $B_n/[P_n,P_n]$, and we characterise the torsion elements of $UVB_n$. 
\end{abstract}

\maketitle

\section{Introduction}
Let $n\geq 1$. The group of unrestricted virtual braids, denoted throughout this paper by $UVB_n$, was introduced by Kauffman and Lambropoulou in~\cite{KaL} as the analogue of fused links in the setting of braids. The classification of fused links is now well known. Such links are distinguished by their \emph{virtual linking number}, see for instance~\cite{N}, where they are considered as the closure of unrestricted virtual braids, as well as~\cite{ABMW} for their classification in terms of Gauss diagrams. Unrestricted virtual braid groups occur as natural quotients of virtual and welded braid groups. They appear for instance in~\cite{KMRW} in the study of local representations of welded braid groups, where they are called \emph{symmetric loop braid groups}, and they may be decomposed as a semi-direct product of a right-angled Artin group, which is in fact the pure subgroup $UVP_n$ of $UVB_n$, by the symmetric group $S_n$~\cite{BBD}. The main aim of this paper is to characterise the torsion elements of $UVB_n$ using this decomposition, namely to show that any element of finite order is a conjugate of an element of $S_n$ by an element of $UVP_n$.

The structure of this paper is as follows. In Section~\ref{sec:uvb}, we give presentations of the virtual braid groups $VB_n$, of the welded braid groups $WB_n$, and of $UVB_n$. In this way, $UVB_n$ may be viewed as a quotient of both $VB_n$ and $WB_n$. We also recall two important results of~\cite{BBD} that describe the structure of the pure unrestricted virtual braid group $UVP_n$ as a direct sum of copies of the free group $F_2$ on two generators, which as mentioned above, allows us to decompose $UVB_n$ as a semi-direct product of $UVP_n$ and $S_n$ in a natural way. A similar decomposition holds for $VB_n$ and $WB_n$, but the canonical homomorphism $\eta \colon\thinspace B_n \to UVB_n$, where $B_n$ is the Artin braid group, is not injective, which is in contrast with the nature of the corresponding homomorphisms for $VB_n$ and $WB_n$. In Section~\ref{sec:cryst}, we study the image of $\eta$, and in Proposition~\ref{prop:cris}, we prove that it is isomorphic to the quotient $B_n/[P_n,P_n]$, where $P_n$ is the pure Artin braid group, and $[P_n,P_n]$ is its commutator subgroup. This quotient has been the subject of recent study, one of the reasons being that it is a crystallographic group~\cite{GGO17,GGO19,GGO20}. The results of~\cite{GGO17} have been generalised to quasi-Abelianised quotients of complex reflection groups~\cite{BM,M16}, and to surface braid groups~\cite{GGOP}. This enables us to give an alternative proof in Proposition~\ref{prop:cris3} and Remark~\ref{rem:uvbnhn} of the fact that $B_n/[P_n,P_n]$ embeds in the semi-direct product $P_n/[P_n,P_n] \rtimes S_n$. In the final section of the paper, Section~\ref{sec:torsionuvbn}, we study the torsion elements of $UVB_n$. We apply the embedding of Proposition~\ref{prop:cris} to give a new combinatorial proof in Theorem~\ref{thm:ord2} of the fact that there are no elements of even order in $B_n/[P_n,P_n]$, and in Theorem~\ref{thm:thm_of_torsion}, we characterise the torsion elements of $UVB_n$ by showing that every such element is conjugate to an element of $S_n$ by an element of $UVP_n$. To our knowledge, it is not known whether a analogous result holds for $VB_n$ and $WB_n$. 

\section{Unrestricted virtual braids}\label{sec:uvb}

In order to define unrestricted virtual braid groups, in this section we recall first the notions of virtual and welded braid groups by exhibiting their usual group presentations. 

\begin{defin}\label{def:vbn}
The \emph{virtual braid group} $VB_n$ is the group defined by the group presentation:

\noindent
Generators: $\{\sig{1}, \ldots, \sig{n-1}, \rr{1}, \ldots, \rr{n-1}\}$

\noindent
Relations:
\begin{align*}
\sig{\ii}  \sig{i+1} \sig{\ii} &= \sig{i+1}\sig{\ii} \sig{i+1} & & \text{for $\ii=1,\dots,\nn-2$} \tag{BR1}\label{eq:br1}\\
\sig{\ii}  \sig{\jj} &= \sig{\jj} \sig{\ii} & & \text{for $\lvert \ii - \jj \rvert \geq 2$} \tag{BR2}\label{eq:br2}\\
\rr\ii \rr{i+1} \rr\ii &= \rr{i+1} \rr\ii\rr{i+1} & & \text{for  $\ii=1,\dots,\nn-2$} \tag{SR1}\label{eq:sr1}\\
\rr\ii \rr\jj &= \rr\jj \rr\ii & & \text{for $\lvert \ii - \jj \rvert \geq 2$}\tag{SR2}\\
\rrq\ii2 &= 1 & & \text{for $\ii=1,\dots,\nno$} \tag{SR3}\label{eq:sr3}\\ 
\sig{i}  \rr{j} &= \rr{j} \sig{i} & & \text{for $\lvert \ii - \jj \rvert \geq 2$}\tag{MR1}\label{eq:mr1}\\
\rr{i}  \rr{i+1} \sig{i} &= \sig{i+1}  \rr{i}   \rr{i+1} & & \text{for $\ii=1,\dots,\nn-2$.} \tag{MR2}\label{eq:mr2}
\end{align*}
\end{defin}

A diagrammatic description of generators and relations of $VB_n$ may be found for instance in~\cite{Bar-0,BP,Kam}. For a topological interpretation, we refer the reader to~\cite{Cis}, and for an algebraic one (in terms of actions on root systems) to~\cite{BPT}. Note that the relations~(\ref{eq:br1})--(\ref{eq:br2}) (resp.\ (\ref{eq:sr1})--(\ref{eq:sr3})) correspond to the usual relations of the Artin braid group $B_{n}$ (resp.\ the symmetric group $S_{n}$) for the set $\{\sig{1}, \ldots, \sig{n-1} \}$ (resp.\ for the set $\{\rr{1}, \ldots, \rr{n-1}\}$), and the remaining relations~(\ref{eq:mr1})--(\ref{eq:mr2}) are `mixed' in the sense that they involve generators of both of these sets.

Recall that the pure braid group $P_n$ is the kernel of the homomorphism $\pi\colon\thinspace B_n\longrightarrow S_n$ defined on the generators $\sig{1}, \ldots, \sig{n-1}$ of $B_{n}$ by $\pi(\sigma_i)=s_i$ for all $i=1,\ldots n-1$, where $s_i$ is the transposition $(i,i+1)$. Analogously, we define the \emph{virtual pure braid group}, denoted by $VP_\nn$, to be the kernel of the homomorphism $\pi_{VP}\colon\thinspace VB_\nn \longrightarrow S_\nn$ that for all $\ii = 1, 2, \dots, n-1$, maps the generators $\sig{\ii}$ and $\rr\ii$ to $s_\ii$. A group presentation for $VP_\nn$ is given in~\cite{Bar-0}. 
Let $\iota\colon\thinspace S_n \longrightarrow VB_n$ be the homomorphism defined by $\iota(s_i) = \rr{i}$ for $i=1, \ldots, n-1$. Since $\iota$ is a section for $\pi_{VP}$, it follows that $\iota$ is injective and that $VB_n$ is a semi-direct product of $VP_n$ by $S_n$. The canonical homomorphism $\eta\colon\thinspace B_n \longrightarrow VB_n$ defined by $\eta(\sigma_i) = \sigma_i$ for $i=1, \ldots, n-1$ is injective~\cite{BP,G,Kam}.  

The welded braid group $WB_n$ may be defined as a quotient of $VB_n$ by adding the following family of relations to the presentation  given in Definition~\ref{def:vbn}:
\begin{equation}
\rr\ii \sig{\ii+1} \sig\ii =  \sig{\ii+1} \sig\ii \rr{\ii+1},\ \text{for $\ii = 1, \dots, \nn-2$,} \tag{\text{OC}}\label{eq:oc}
\end{equation}
where OC stands for `Over Commuting'. Welded braid groups have several different equivalent definitions~\cite{D}. In particular, they may be defined as \emph{basis conjugating automorphisms} of free groups \cite{FRR}, from which it follows that the homomorphism $\eta\colon\thinspace B_n \longrightarrow WB_n$ defined by $\eta(\sigma_i) = \sigma_i$ for $i=1, \ldots, n-1$ is injective~\cite{Kam}. As in the case of virtual braid groups, the \emph{welded pure braid group}, denoted by $WP_\nn$, is defined to be the kernel of the homomorphism $\pi_{WP}\colon\thinspace WB_\nn \longrightarrow S_\nn$ given by sending the generators $\sig{\ii}$ and $\rr\ii$ of $WB_\nn$ to $s_\ii$ for all $\ii = 1, 2, \dots, n-1$. A  group presentation for $WP_\nn$ may be found in~\cite{BH,D,Su}. Abusing notation, we define $\iota\colon\thinspace S_n \longrightarrow WB_n$ to be also the homomorphism defined by $\iota(s_i) = \rr{i}$ for $i=1, \ldots, n-1$. As in the virtual case, $\iota$ is a section for $\pi_{WP}$, and therefore $\iota$ is injective and $WB_n$ is a semi-direct product of $WP_n$ by $S_n$. Note that the symmetrical relations:
\begin{equation}
\rr{\ii+1} \sig{\ii} \sig{\ii+1} =  \sig\ii \sig{\ii+1}  \rr{\ii},\ \text{for $\ii = 1, \dots, \nn-2$},  \tag{\text{UC}}\label{eq:uc}
\end{equation}
where UC stands for `Under Commuting', do not hold in~$\WB\nn$ (see for instance~\cite{BS,Kam}). Consequently, the following group, which was first defined by Kauffman and Lambropoulou in~\cite{KaL}, is a proper quotient of $WB_n$. 

\begin{defin}\label{D:Unrestricted}
The \emph{unrestricted virtual braid group} $UVB_n$ is the group defined by the following group presentation:

\noindent
Generators: $\{\sig{1}, \ldots, \sig{n-1}, \rr{1}, \ldots, \rr{n-1}\}$

\noindent
Relations: the seven relations~(\ref{eq:br1})--(\ref{eq:br2}), (\ref{eq:sr1})--(\ref{eq:sr3}) and~(\ref{eq:mr1})--(\ref{eq:mr2}) of  Definition~\ref{def:vbn}, plus the relations of types~(\ref{eq:oc}) and~(\ref{eq:uc}).
 \end{defin}

As in the case of $VB_n$, let $\pi_{UVP} \colon\thinspace UVB_n\longrightarrow S_n$ be the homomorphism defined by $\pi_{UVP}(\sigma_i)=\pi_{UVP}(\rho_i)=s_i$ for $i=1,\ldots,n-1$. The kernel of $\pi_{UVP}$ is the \emph{unrestricted virtual pure braid group}, denoted by $UVP_n$. For $1\leq i,j\leq n$, $i\neq j$, we define the elements $\lambda_{i,j}$ of $UVP_n$ as follows:
\begin{equation}\label{eq:lambda}
\begin{cases}
\lambda_{i,i+1}=\rho_i\sigma_i^{-1} & \text{for $i=1,\ldots,n-1$}\\
\lambda_{i+1,i}=\sigma_i^{-1}\rho_i & \text{for $i=1,\ldots,n-1$}\\
\lambda_{i,j}=\rho_{j-1}\rho_{j-2}\cdots\rho_{i+1}\lambda_{i,i+1} \rho_{i+1}\cdots\rho_{j-2} \rho_{j-1}& \text{for $1\leq i<j-1\leq n-1$}\\
\lambda_{j,i}=\rho_{j-1}\rho_{j-2}\cdots\rho_{i+1}\lambda_{i+1,i}\rho_{i+1}\cdots\rho_{j-2}\rho_{j-1}& \text{for $1\leq i<j-1\leq n-1$.}
\end{cases}
\end{equation}

\begin{thm}[Bardakov--Bellingeri--Damiani~\cite{BBD}]\label{thm:tss}
The group $UVP_n$ admits the following presentation:\\
Generators: $\lambda_{i,j}$, where $1\leq i,j\leq n$, $i\neq j$.\\
Relations: The generators commute pairwise except for the pairs $\lambda_{i,j}$ and $\lambda_{j,i}$ for all $1\leq i,j\leq n$, $i\neq j$.
\end{thm}

It follows from Theorem~\ref{thm:tss} that the group $UVP_n$ is a right-angled Artin group that is isomorphic to the direct product of the ${n(n-1)/2}$ free groups $F_{i,j}$ of rank $2$, where $1\leq i<j\leq n$, and where $\{ \lambda_{i,j}, \lambda_{j,i} \}$ is a basis of $F_{i,j}$. By convention, if $1\leq i,j\leq n$ then we set $F_{j,i}=F_{i,j}$. As for $VB_n$, let $\iota\colon\thinspace S_n \longrightarrow UVB_n$  be the homomorphism defined by $\iota(s_i) = \rr{i}$ for $i=1, \ldots, n-1$. Then $\iota$ is injective, and is a section for $\pi_{UVP}$. This leads to the following natural decomposition of $UVB_n$.

\begin{thm}[Bardakov--Bellingeri--Damiani~\cite{BBD}]\label{thm:tss2}
The group $UVB_n$ is isomorphic to the semi-direct product $UVP_n\rtimes S_n$, where $S_n$ acts on $UVP_n$ by permuting the indices of the elements of the generating set of $UVP_n$ given in Theorem~\ref{thm:tss}. More precisely, for all $\lambda_{i,j} \in UVP_n$, where $1\leq i, j \leq n$, $i\neq j$, and for all $s\in S_n$ we have:
\begin{equation}\label{eq:conjuvpn}
\iota(s)\,\lambda_{i,j}\,\iota(s)^{-1}=\lambda_{s(i),s(j)}.
\end{equation}
\end{thm} 

\section{Crystallographic groups}\label{sec:cryst}

As we mentioned in Section~\ref{sec:uvb}, the Artin braid group $B_n$ embeds naturally in $VB_n$ and $WB_n$ via the injective homomorphism $\eta$ defined in each of these two cases. However this is no longer the case for the canonical homomorphism $\eta\colon\thinspace B_n \longrightarrow UVB_n$ defined by $\eta(\sigma_i) = \sigma_i$ for $i=1, \ldots, n-1$. In this section, we show that the image of $B_n$ under the homomorphism  $\eta$ is isomorphic to the quotient of $B_n$ by the commutator subgroup $[P_n, P_n]$. This quotient was introduced by Tits in~\cite{Ti} as \emph{groupes de Coxeter \'etendus}, and has been studied more recently by various authors (see for example~\cite{BM,Di,GGO17,GGO19,GGO20,M16}).

\begin{prop} \label{prop:cris}
Let $\eta\colon\thinspace B_n\longrightarrow UVB_n$ be the canonical homomorphism defined by $\eta(\sigma_i)=\sigma_i$ for $1\leq i\leq n-1$. Then $\eta(B_n)$ is isomorphic to the group $B_n/[P_n, P_n]$.
\end{prop} 

\begin{proof}
Consider the homomorphism $\eta\colon\thinspace B_n\longrightarrow UVB_n$. Since $P_n=\ker(\pi)$ (resp.\ $UVP_n=\ker(\pi_{UVP})$), where $\pi\colon\thinspace B_n\longrightarrow S_n$ (resp.\ $\pi_{UVP}\colon\thinspace UVB_n \longrightarrow S_n$) is defined by $\pi(\sigma_i)=s_i$ (resp.\ $\pi_{UVP}(\sigma_i)= \pi_{UVP}(\rho_i)=s_i$) for all $1\leq i\leq n-1$, it follows from the definition of 
$\eta\colon\thinspace B_n\longrightarrow UVB_n$ that $\pi=\pi_{UVP} \circ \eta$, and thus $\pi=\pi_{UVP}\left\lvert_{\eta(B_{n})}\right. \circ \eta$. Since $\pi$ is surjective,  the homomorphism $\pi_{UVP}\left\lvert_{\eta(B_{n})}\right.\colon\thinspace \eta(B_{n})\longrightarrow S_{n}$ is too. From the equality $\pi=\pi_{UVP} \circ \eta$, we obtain the following commutative diagram of short exact sequences:
\begin{equation}\label{eq:cd1}
\begin{tikzcd}
1 \arrow[r] & P_{n} \arrow[r] \arrow[d, "\eta\left\lvert_{P_n}\right."] & B_{n} \arrow[r, "\pi"] \arrow[d, "\eta"] & S_{n} \arrow[r] \ar[equal]{d} & 1\\
1 \arrow[r] & UVP_{n} \arrow[r] & UVB_{n} \arrow[r, "\pi_{UVP}"] & S_{n} \arrow[r] & 1.
\end{tikzcd}
\end{equation}
We claim that $\eta(P_{n})=\ker(\pi_{UVP}\left\lvert_{\eta(B_{n})}\right.)$. To prove the claim, the exactness of~(\ref{eq:cd1}) implies that $\eta(P_{n}) \subset \ker(\pi_{UVP}\left\lvert_{\eta(B_{n})}\right.)$. Conversely, let $x\in \ker(\pi_{UVP}\left\lvert_{\eta(B_{n})}\right.)$. Then $x\in \eta(B_{n})$, and there exists $y\in B_{n}$ such that $\eta(y)=x$. The commutativity of~(\ref{eq:cd1}) implies that $y\in P_{n}$, and so $x\in \eta(P_{n})$, so $\ker(\pi_{UVP}\left\lvert_{\eta(B_{n})}\right.) \subset \eta(P_{n})$, and the claim follows.

Considering $\eta$ and its restriction to $P_{n}$, we obtain the following commutative diagram of short exact sequences:
\begin{equation}\label{eq:cdeta}
\begin{tikzcd}
 &&1 \arrow[d] & 1 \arrow[d]\\
1 \arrow[r] & \ker(\eta\left\lvert_{P_n}\right.) \arrow[d] \arrow[r] & P_n \arrow[d] \arrow[r, "\eta\left\lvert_{P_n}\right."] & \eta(P_n)\arrow[r] \arrow[d] \arrow[r] & 1 \\
1 \arrow[r] & \ker(\eta)  \arrow[r] &B_n \arrow[d, "\pi"]\arrow[r, "\eta"] & \eta(B_n)\arrow[d, "\pi_{UVP}\left\lvert_{\eta(B_{n})}\right."]\ar[r] & 1.\\
&&S_n  \arrow[d] \ar[equal]{r} & S_n  \arrow[d]\\
&&1 & 1
\end{tikzcd}
\end{equation}
Note that the exactness of the rightmost column of~(\ref{eq:cdeta}) follows from the claim of the previous paragraph.
By a standard diagram-chasing argument, we conclude from~(\ref{eq:cdeta}) that $\ker(\eta)= \ker(\eta\left\lvert_{ P_n}\right.)$. 
Applying the isomorphism $UVB_n \cong UVP_n\rtimes S_n$ of Theorem~\ref{thm:tss2} and using equation~(\ref{eq:lambda}), we see that  $\eta(\sigma_i)=\lambda_{\ii, \ii+1}\inv \rho_{i}$. Recall that the pure braid group $P_n$ is generated by the set $\{a_{ij} \mid 1 \leq i < j \leq \nn\}$, where:
\begin{gather*}
a_{\ii,\ii+1} = \sigg\ii{2}, \\
a_{\ii, \jj} = \sig{\jj-1} \sig{\jj-2} \cdots \sig{\ii+1} \sigg\ii{2} \siginv{\ii+1} \cdots \siginv{\jj-2} \siginv{\jj-1} \mbox{,} \quad  \mbox{for } \ii+1 < \jj \leq \nn.
\end{gather*}

In~\cite[Corollary 2.8]{BBD} it was shown that 
$\eta(a_{\ii, \jj})=\lambda_{\ii, \jj}^{-1} \lambda_{\jj, \ii}^{-1}$  for $1 \leq \ii+1 < \jj \leq \nn$. From the construction, we deduce that $[P_n, P_n]$ is contained in $\ker(\eta\left\lvert_{ P_n}\right.)$. Further, $\eta(P_n)$ is isomorphic to $P_n/[P_n,P_n]=\Z^{n(n-1)/2}$ (see~\cite[Corollary 2.8]{BBD}). It follows that  $[P_n, P_n]$ coincides with $\ker(\eta\left\lvert_{ P_n}\right.)$. The statement follows by combining this with the fact that $\ker(\eta)= \ker(\eta\left\lvert_{ P_n}\right.)$ and using the exactness of~(\ref{eq:cdeta}).
\end{proof}

\begin{rem}
We can rephrase Proposition~\ref{prop:cris} by saying that the canonical homomorphism $\eta\colon\thinspace B_n \to UVB_n$
factors through an embedding of $B_n/[P_n, P_n]$ into $UVB_n$.
On the other hand, by~\cite[Theorem~2.5]{CO}, the canonical homomorphism from $B_n$ to $VB_n$ induces an embedding of $B_n/[P_n, P_n]$ in $VB_n/[VP_n,VP_n]$. Since $VB_n/[VP_n,VP_n]$ is isomorphic to $UVB_n/[UVP_n,UVP_n]$~\cite[Theorem~4.1]{CO}, the canonical homomorphism  $B_n$ to $UVB_n$ gives rise to an embedding of the quotient $B_n/[P_n, P_n]$ in $UVB_n/[UVP_n,UVP_n]$. We therefore conclude that the composition of the embedding of $B_n/[P_n, P_n]$ in $UVB_n$ defined in the proof of Proposition~\ref{prop:cris} and the canonical projection $q_n\colon\thinspace UVB_n \to UVB_n/[UVP_n,UVP_n]$ is also injective. 
\end{rem}
 
We recall that an abstract group $\Gamma$ is said to be \emph{crystallographic} if it can be realised as an extension of a free Abelian subgroup of $\Gamma$ of maximal rank by a finite group~\cite[Theorem 2.1.4]{Dek}. Using the following result, in~\cite{GGO17}, Gon\c{c}alves--Guaschi--Ocampo proved that the group $B_n/[P_n, P_n]$ is a crystallographic group. 
 
\begin{prop}[Gon\c{c}alves--Guaschi--Ocampo~\cite{GGO17}]\label{prop:ggo}
Let $n\geq 2$. Then the following sequence:
\begin{equation}\label{eq:cryst}
\begin{tikzcd}
1 \arrow[r] & \mathbb{Z}^{n(n-1)/2} \arrow[r] &B_n/[P_n, P_n] \arrow[r, "\widehat{\pi}"] & S_n \ar[r] & 1,
\end{tikzcd}
\end{equation}
is short exact, where $\widehat{\pi}$ is the homomorphism induced by $\pi\colon\thinspace B_n\longrightarrow S_n$.
\end{prop}

Note that the sequence~(\ref{eq:cryst}) does not split~\cite{Di,GGO17}. It follows from Propositions~\ref{prop:cris} and~\ref{prop:ggo} that the group $\eta(B_n)$ is a crystallographic subgroup of $UVB_n$.
We remark also that $\operatorname{\text{Im}}(\eta)$ is not normal in $UVB_n$, and that the normal closure of $\operatorname{\text{Im}}(\eta)$ is of index $2$ in $UVB_n$. This follows from the fact that in the quotient of $UVB_n$ by $\operatorname{\text{Im}}(\eta)$, the generators $\rho_i$ are  identified to a single element due to relations~(\ref{eq:oc}) and~(\ref{eq:uc}) of Definition~\ref{D:Unrestricted}. The group $UVB_n$ contains other  crystallographic groups of rank $n(n-1)/2$  that are not isomorphic to $\eta(B_n)$. One such example is given in the following proposition that is a straightforward consequence of Theorems \ref{thm:tss} and \ref{thm:tss2}

\begin{prop} \label{prop:cris2}
Let $C_n$ be the subgroup of $UVB_n$ generated by $\{ \lambda_{i,j} \lambda_{j,i}^{-1} \}$ for $1\leq i<j\leq n$ and by $\rho_i$ for $i=1, \ldots, n$, and let $H_n$ be the subgroup of $C_n$ generated by $\{  \lambda_{i,j} \lambda_{j,i}^{-1} \}$ for $1\leq i<j\leq n$.
The restriction of $\pi_{UVP}$ to $C_n$ defines the following split exact sequence: 
\begin{equation*}
\begin{tikzcd}
1 \arrow[r] & H_n \arrow[r] &C_n \arrow[r, "\pi_{UVP}{\big|}_{C_n}"] & S_n \ar[r] & 1,
\end{tikzcd}
\end{equation*}
In particular $C_n$ is crystallographic. 
 \end{prop}


\begin{prop} \label{prop:cris3}
Let $H_n$ be the subgroup of $UVB_n$ defined in the statement of  Proposition~\ref{prop:cris2}, let $\langle \! \langle H_n \rangle \! \rangle$ its normal closure in $UVB_n$, and let $\{x_{i,j}\}_{1\le i<j\le n}$ be a set of generators of $\mathbb{Z}^{n(n-1)/2}$.
\begin{enumerate}[(a)]
\item\label{it:cris3a} The quotient $UVB_n/\langle \! \langle H_n \rangle \! \rangle$ is isomorphic to the semi-direct product $\mathbb{Z}^{n(n-1)/2} \rtimes S_n$,
where for any $s \in S_n$, $s \cdot x_{i,j}= x_{s(i),s(j)}$, where we take $x_{i,j}=x_{j,i}$.

\item\label{it:cris3b} Let $\theta\colon\thinspace B_n \longrightarrow UVB_n/\langle \! \langle H_n \rangle \! \rangle$ be the composition of $ \eta: B_n \longrightarrow UVB_n$ and the projection $\pi_H\colon\thinspace UVB_n \longrightarrow UVB_n/\langle \! \langle H_n \rangle \! \rangle$. The group $\theta(B_n)$ is isomorphic to $B_n/[P_n, P_n]$.
 \end{enumerate}
 \end{prop}
 
\begin{proof} 
Part~(\ref{it:cris3a}) follows from Theorems~\ref{thm:tss} and~\ref{thm:tss2} that imply that $UVB_n/\langle \! \langle H_n \rangle \! \rangle = UVP_n/\langle \! \langle H_n \rangle \! \rangle \rtimes S_n$. The identification $\lambda_{i,j} = \lambda_{j,i}$ for $i \not= j$ implies that $UVP_n/H_n=\mathbb{Z}^{n(n-1)/2} $, and it also defines the induced action of $S_n$ on this quotient. The proof of Proposition~\ref{prop:cris} can be adapted to prove part~(\ref{it:cris3b}) (the image of $P_n$ by $\theta$ is a free Abelian group of rank $n(n-1)/2$, and is thus isomorphic to $P_n/[P_n,P_n]$). 
\end{proof}

\begin{rem}\label{rem:uvbnhn}
The quotient $UVB_n/\langle \!\langle H_n \rangle \!\rangle$ coincides with the  semi-direct product considered in~\cite[Section 2.7]{Ti} and in the Remark following~\cite[Proposition 5.12]{M09}. Thus Proposition~\ref{prop:cris3} yields a new proof of the fact that the crystallographic braid group $B_n/[P_n, P_n]$ embeds in the semi-direct product $\mathbb{Z}^{n(n-1)/2} \rtimes S_n$. 
\end{rem}

\section{Torsion elements of $UVB_n$}\label{sec:torsionuvbn}

Let $n\geq 3$. In~\cite{GGO17}, it was shown that $B_n/[P_n, P_n]$ has torsion, but that it possesses no elements of even order. Viewing $B_n/[P_n, P_n]$ as the subgroup $\operatorname{\text{Im}}(\eta)$ of $UVB_n$ via Proposition~\ref{prop:cris}, we may give an alternative proof of this latter fact. We do this by first characterising the elements of $UVB_n$ of order $2$, and then showing that these elements do not belong to $\operatorname{\text{Im}}(\eta)$.

\begin{thm}\label{thm:ord2}
Let $n\geq 3$.
\begin{enumerate}[(a)]
\item\label{it:ord2a} An element $v\in UVB_{n}$ is of order $2$ if and only if there exist $\rho\in \operatorname{\text{Im}}(\iota)$ of order $2$ and $g\in UVP_{n}$ such that $v=g\rho g^{-1}$.
\item\label{it:ord2b} The elements of order $2$ of $UVB_n$ do not belong to $\operatorname{\text{Im}}(\eta)$. In particular, $B_n/[P_n, P_n]$ has no elements of even order.
\end{enumerate}
\end{thm}

\begin{proof}
Let $n\geq 3$. 
\begin{enumerate}[(a)]
\item First observe that the given condition is clearly sufficient. Conversely, suppose that $v\in UVB_{n}$ is of order $2$. Identifying $UVB_n$ with the internal semi-direct product $UVP_n\rtimes \operatorname{\text{Im}}(\iota)$ using Theorem~\ref{thm:tss2}, there exist unique elements $w\in UVP_{n}$ and $\rho \in \operatorname{\text{Im}}(\iota)$ such that $v=w\rho$, and the pair $(w,\rho)$ is non trivial. Further, the fact that $UVP_{n}$ is torsion free implies that $\rho \neq 1$. Since $v$ is of order $2$, we have $1=v^{2}=(w\rho)^{2}=w\ldotp \rho w \rho\inv \ldotp \rho^{2}$, where $w\ldotp \rho w \rho\inv\in UVP_n$ and $\rho^{2}\in \operatorname{\text{Im}}(\iota)$, from which we deduce using Theorem~\ref{thm:tss2} that $\rho$ is also of order $2$. Hence $\pi_{UVP}(\rho) \in S_{n}$ is also of order $2$, and thus may be written as a non-trivial product of disjoint transpositions. So there exists $\tau\in \operatorname{\text{Im}}(\iota)$ such that $\pi_{UVP}(\tau\rho\tau\inv)=\sigma$, where $\sigma=(1,2)(3,4)\cdots (m,m+1)$, and $1\leq m\leq n-1$ is odd. Setting $\widetilde{v}=\tau v \tau\inv$, $\widetilde{w}=\tau w \tau\inv$ and $\widetilde{\rho}=\tau \rho \tau\inv$, it follows that $\widetilde{v}= \widetilde{w} \widetilde{\rho}$, where $\widetilde{v}\in UVB_n$ is of order $2$, $\widetilde{w}\in UVP_{n}$, $\widetilde{\rho} \in \operatorname{\text{Im}}(\iota)$ is of order $2$, $\pi_{UVP}(\widetilde{\rho})= \sigma$, and $\widetilde{w}\ldotp \widetilde{\rho}\widetilde{w}\widetilde{\rho}\inv=1$. For $1\leq i<j\leq n$, 
let $\pi_{i,j}\colon\thinspace UVP_{n} \longrightarrow F_{i,j}$ denote the projection of $UVP_{n}$ onto $F_{i,j}$ (see Theorem~\ref{thm:tss} and the comments that follow it). Applying~(\ref{eq:conjuvpn}), we observe that conjugation by an element of $UVB_n$ permutes the $\{ F_{i,j}\}_{1\leq i<j\leq n}$. Let $T$ denote the subset of $n(n-1)/2$ transpositions of $S_{n}$, and let $T_{\sigma}=\{ (i,i+1)\in T  \,\mid\, i\in \{ 1,3,\ldots,m\}\}$. Then the subgroup $\langle\sigma\rangle$ of $S_{n}$ of order $2$ acts on $T$ by conjugation, and if $1\leq i<j\leq n$,  the orbit $\mathcal{O}(i,j)$ of $(i,j)$ is equal to $\{(i,j),(\sigma(i),\sigma(j))\}$, which is equal to $\{ (i,j)\}$ if and only if either $(i,j)\in T_{\sigma}$ or if $m+2\leq i<j\leq n$, and contains two elements otherwise. Further, for all $1\leq i<j\leq n$, $\sigma(i)<\sigma(j)$ if and only if $(i,j)\notin T_{\sigma}$. Let $\mathcal{T}$ be a transversal for this action of $\langle\sigma\rangle$ on $T$. Using Theorem~\ref{thm:tss}, we obtain:
\begin{equation}\label{eq:tildew}
\widetilde{w}=\prod_{(i,j)\in \mathcal{T}}\prod_{(k,l)\in \mathcal{O}(i,j)} w_{k,l}(\lambda_{k,l}, \lambda_{l,k}),
\end{equation}
where $w_{k,l}=w_{k,l}(\lambda_{k,l}, \lambda_{l,k})\in F_{k,l}$, and this decomposition is unique up to permutation of the factors. So:
\begin{align}
1&= \widetilde{w}\ldotp \widetilde{\rho} \widetilde{w} \widetilde{\rho}\,\inv=\left( \prod_{(i,j)\in \mathcal{T}}\prod_{(k,l)\in \mathcal{O}(i,j)} w_{k,l}(\lambda_{k,l}, \lambda_{l,k})\right) \widetilde{\rho} \left( \prod_{(i,j)\in \mathcal{T}}\prod_{(k,l)\in \mathcal{O}(i,j)} w_{k,l}(\lambda_{k,l}, \lambda_{l,k})\right)\widetilde{\rho}\,\inv\notag\\
&= \left( \prod_{(i,j)\in \mathcal{T}}\prod_{(k,l)\in \mathcal{O}(i,j)} w_{k,l}(\lambda_{k,l}, \lambda_{l,k})\right) \left( \prod_{(i,j)\in \mathcal{T}}\prod_{(k,l)\in \mathcal{O}(i,j)} w_{k,l}(\lambda_{\sigma(k),\sigma(l)}, \lambda_{\sigma(l), \sigma(k)})\right)\notag\\
&= \prod_{(i,j)\in \mathcal{T}}\prod_{(k,l)\in \mathcal{O}(i,j)} w_{k,l}(\lambda_{k,l}, \lambda_{l,k}) \ldotp w_{\sigma(k),\sigma(l)}(\lambda_{k,l}, \lambda_{l,k}),\label{eq:wrho}
\end{align}
where we have used the fact that $\sigma^{2}=1$. Note also that if $(i,j)\in T_{\sigma}$ then $j=\sigma(i)>\sigma(j)=i$, and in this case, any term of the form $w_{\sigma(i),\sigma(j)}(\lambda_{i,j}, \lambda_{j,i})$ in~(\ref{eq:wrho}) should be interpreted as $w_{i,j}(\lambda_{j,i}, \lambda_{i,j})$. The expression~(\ref{eq:wrho}) is written with respect to the direct product structure of $UVP_{n}$. It follows from Theorem~\ref{thm:tss} that $w_{i,j}(\lambda_{i,j}, \lambda_{j,i}) \ldotp w_{\sigma(i),\sigma(j)}(\lambda_{i,j}, \lambda_{j,i})=1$ for all $1\leq i< j\leq n$, and hence:
\begin{equation}\label{eq:wsigmaij}
w_{\sigma(i),\sigma(j)}(\lambda_{\sigma(i),\sigma(j)}, \lambda_{\sigma(j),\sigma(i)})= w_{i,j}^{-1}(\lambda_{\sigma(i),\sigma(j)}, \lambda_{\sigma(j),\sigma(i)}).
\end{equation}


Suppose that $(i,j)\in T_{\sigma}$. Then $w_{i,j}(\lambda_{i,j}, \lambda_{j,i}) \ldotp w_{i,j}( \lambda_{j,i},\lambda_{i,j})=1$. Writing $a=\lambda_{i,j}$, $b=\lambda_{j,i}$ and $w_{i,j}(a,b)=a^{k_{1}}b^{l_{1}}\cdots a^{k_{m}}b^{l_{m}}$, where $m\geq 0$, $k_{1},l_{1},\ldots, k_{m},l_{m}\in \Z$ and $l_{1},\ldots, k_{m}\neq 0$, it follows from the relation $w_{i,j}(a,b)\ldotp w_{i,j}(b,a)=1$ that $j_{q}=-k_{m+1-q}$ for all $q=1,\ldots,m$. Taking $y_{i,j}(a,b)=a^{k_{1}}b^{-k_{m}}\cdots a^{k_{m/2}}b^{-k_{(m+2)/2}}$ (resp.\ $y_{i,j}(a,b)=a^{k_{1}}b^{-k_{m}}\cdots a^{k_{(m-1)/2}}b^{-k_{(m+3)/2}} a^{k_{(m+1)/2}}$) if $m$ is even (resp.\ odd), we see that $w_{i,j}(a,b)=y_{i,j}(a,b) \ldotp y_{i,j}\inv(b,a)$. Thus $w_{i,j}(\lambda_{i,j}, \lambda_{j,i})= y_{i,j}(\lambda_{i,j}, \lambda_{j,i}) \ldotp y_{i,j}^{-1}(\lambda_{j,i},\lambda_{i,j})$. 
Setting $z_{i,j}=w_{i,j}$ if $(i,j)\notin T_{\sigma}$ and $z_{i,j}=y_{i,j}$ if $(i,j)\in T_{\sigma}$, it follows from~(\ref{eq:tildew}) and~(\ref{eq:wsigmaij}) that:
\begin{equation}\label{eq:tildev}
\widetilde{v}=\widetilde{w}\widetilde{\rho}=\left(\prod_{(i,j)\in \mathcal{T}} z_{i,j}(\lambda_{i,j}, \lambda_{j,i}) z_{i,j}^{-1}(\lambda_{j,i}, \lambda_{i,j}) \right)\widetilde{\rho}.
\end{equation}
Now the terms $z_{i,j}(\lambda_{i,j}, \lambda_{j,i})$ (resp.\ $z_{i,j}^{-1}(\lambda_{j,i}, \lambda_{i,j})$) appearing in~(\ref{eq:tildev}) commute pairwise, and setting $\widetilde{g}=\prod_{(i,j)\in \mathcal{T}} z_{i,j}(\lambda_{i,j}, \lambda_{j,i})\in UVP_{n}$, we obtain:
\begin{equation*}
\widetilde{v}=\left(\prod_{(i,j)\in \mathcal{T}} z_{i,j}(\lambda_{i,j}, \lambda_{j,i}) \right)\left(\prod_{(i,j)\in \mathcal{T}} z_{i,j}^{-1}(\lambda_{j,i}, \lambda_{i,j}) \right)\widetilde{\rho}=\widetilde{g} \widetilde{\rho}\, \widetilde{g}^{-1}.
\end{equation*}
Hence $v=\tau^{-1} \widetilde{v} \tau=g \rho g^{-1}$, where $\rho \in \operatorname{\text{Im}}(\iota)$ is of order $2$ and $g=\tau^{-1}\widetilde{g}\tau \in UVP_{n}$, which proves that the condition of part~(\ref{it:ord2a}) is also necessary.

\item We start by characterising the elements of $\operatorname{\text{Im}}(\eta)$. Let $v\in UVB_n$. Then $\pi_{UVP}(v)\in S_n$, and so $\pi_{UVP}(v)=s_{i_1}\cdots s_{i_r}$, where $r\geq 0$, and for all $j=1,\ldots,r$, $s_{i_j}\in \{ (k,k+1) \,\mid\, k=1,\ldots, n-1\}$. Hence $v=w\rho$, where $w\in UVP_n$ and $\rho= \iota(\pi_{UVP}(v))=\rho_{i_1}\cdots \rho_{i_r}$. Let $\beta=\sigma_{i_1}\cdots \sigma_{i_r}\in B_n$. Then by~(\ref{eq:lambda}), we have $\eta(\beta)=\sigma_{i_1}\cdots \sigma_{i_r}=\lambda_{i_1,i_1+1}^{-1}\rho_{i_1} \cdots \lambda_{i_r,i_r+1}^{-1}\rho_{i_r}=y \rho_{i_1}\cdots \rho_{i_r}=y \rho \in \operatorname{\text{Im}}(\eta)$, where $y\in UVP_n$. It follows that:
\begin{equation*}
v\in \operatorname{\text{Im}}(\eta) \Longleftrightarrow (\eta(\beta))^{-1} v\in \operatorname{\text{Im}}(\eta) \Longleftrightarrow \rho^{-1} y^{-1} w \rho \in \operatorname{\text{Im}}(\eta).
\end{equation*}
Now $y^{-1} w \in UVP_n$, so $\rho^{-1} y^{-1} w \rho \in UVP_n$, and it follows that $v\in \operatorname{\text{Im}}(\eta)$ if and only if $\rho^{-1} y^{-1} w \rho\in \eta(P_n)$. From Section~\ref{sec:cryst}, $P_n$ is generated by the set $\{a_{ij} \,\mid\, 1 \leq i < j \leq n\}$, and since $\eta(a_{ij})=\lambda_{i,j}^{-1} \lambda_{j,i}^{-1}$ for all $1\leq i<j\leq n$, we see that $v\in \operatorname{\text{Im}}(\eta)$ if and only if $\rho^{-1} y^{-1} w \rho$ belongs to the free Abelian subgroup $\Gamma$ of $UVP_n$ of rank $n(n-1)/2$ generated by the set $\{\lambda_{i,j}^{-1} \lambda_{j,i}^{-1} \,\mid\, 1 \leq i < j \leq n\}$. In particular, if $1\leq i<j\leq n$ and $\varepsilon_{i,j}\colon\thinspace UVP_n \longrightarrow \Z$ denotes the evaluation homomorphism defined by $\varepsilon_{i,j}(\lambda_{k,l})=1$ if $\{k,l\}=\{i,j\}$ and $\varepsilon_{i,j}(\lambda_{k,l})=0$ otherwise, where $1\leq k,l\leq n$, $k\neq l$, then $\varepsilon_{i,j}(\Gamma)=2\Z$ for all $1\leq i<j\leq n$.

Now suppose on the contrary that $\operatorname{\text{Im}}(\eta)$ possesses an element $v$ of order $2$. From part~(\ref{it:ord2a}), there exist $g\in UVP_n$ and $\rho\in \operatorname{\text{Im}}(\iota)$ of order $2$ such that $v=g\rho g^{-1}$. Then $\pi_{UVP}(\rho)$ is also of order $2$, and so may be written as a non-trivial product of transpositions whose supports are pairwise disjoint. So there exist $1\leq r\leq n/2$ and distinct elements $i_1,j_1,\ldots,i_r,j_r$ of $\{1,\ldots,n\}$ such that $i_k<j_k$ for all $1\leq k \leq r$ and $\rho=\iota((i_1,j_1)\cdots (i_r,j_r))$. Let $\tau \in \operatorname{\text{Im}}(\iota)$ be such that $\pi_{UVP}(\tau \rho\tau^{-1})=(1,2) \cdots (2r-1,2r)$, let $\widetilde{\rho}= \tau \rho\tau^{-1} \in \operatorname{\text{Im}}(\iota)$, let $\widehat{\tau}\in B_n$ be such that $\pi(\widehat{\tau})= \pi_{UVP}(\tau)$, let $\beta = \sigma_1\sigma_3 \cdots \sigma_{2r-1}$, and let $\widehat{\beta} = \widehat{\tau}^{-1} \beta \widehat{\tau}$. Then $\widetilde{\rho}=\rho_1\rho_3 \cdots \rho_{2r-1}$ and  $\eta(\beta)=\sigma_1\sigma_3 \cdots \sigma_{2r-1}=\lambda_{1,2}^{-1}\lambda_{3,4}^{-1}\cdots \lambda_{2r-1,2r}^{-1} \widetilde{\rho}$ by~(\ref{eq:lambda}) and Theorem~\ref{thm:tss}. Since $v\in \operatorname{\text{Im}}(\eta)$, it follows that $(\eta(\widehat{\beta}))^{-1}\ldotp v \in \operatorname{\text{Im}}(\eta)$. Hence:
\begin{equation}\label{eq:etav}
(\eta(\widehat{\beta}))^{-1}\ldotp v = \eta(\widehat{\tau}^{-1})  \widetilde{\rho}^{\,-1} \eta(\widehat{\tau}) \ldotp \eta(\widehat{\tau}^{-1})  \lambda_{2r-1,2r} \cdots \lambda_{3,4} \lambda_{1,2} \eta(\widehat{\tau}) \ldotp g\rho g^{-1}.
\end{equation}
Now using~(\ref{eq:cd1}), we have $\pi_{UVP}(\eta(\widehat{\tau}))= \pi(\widehat{\tau})=\pi_{UVP}(\tau)$, so:
\begin{equation}\label{eq:etaconj}
\eta(\widehat{\tau}^{-1})  \lambda_{2r-1,2r} \cdots \lambda_{3,4} \lambda_{1,2} \eta(\widehat{\tau})= \lambda_{i_r,j_r} \cdots \lambda_{i_2,j_2} \lambda_{i_1,j_1}.
\end{equation}
It follows from~(\ref{eq:etav}) and~(\ref{eq:etaconj}) that:
\begin{align}
(\eta(\widehat{\beta}))^{-1}\ldotp v &=\eta(\widehat{\tau}^{-1})  \widetilde{\rho}^{\,-1} \eta(\widehat{\tau}) \rho\ldotp \rho^{-1} \lambda_{i_r,j_r} \cdots \lambda_{i_2,j_2} \lambda_{i_1,j_1} \rho\ldotp \rho^{-1} g\rho g^{-1}\notag\\
&=\eta(\widehat{\tau}^{-1})  \widetilde{\rho}^{\,-1} \eta(\widehat{\tau}) \rho  \lambda_{j_r,i_r} \cdots \lambda_{j_2,i_2} \lambda_{j_1,i_1} \rho^{-1} g\rho \ldotp g^{-1}\label{eq:etav2}
\end{align}
Since $\pi_{UVP}(\eta(\widehat{\tau}))= \pi(\widehat{\tau})=\pi_{UVP}(\tau)$, there exists $z\in UVP_n$ such that $\eta(\widehat{\tau})=\tau z$, and hence:
\begin{align}
(\eta(\widehat{\beta}))^{-1}\ldotp v &= z^{-1} \tau^{-1} \widetilde{\rho}^{\,-1} \tau z \rho \ldotp \lambda_{j_r,i_r} \cdots \lambda_{j_2,i_2} \lambda_{j_1,i_1} \ldotp \rho^{-1} g\rho \ldotp g^{-1}\notag\\
&=z^{-1}\ldotp \rho^{-1} z \rho \ldotp \lambda_{j_r,i_r} \cdots \lambda_{j_2,i_2} \lambda_{j_1,i_1} \ldotp \rho^{-1} g\rho \ldotp g^{-1}.\label{eq:prodzg}
\end{align}
The expression on the right-hand side of~(\ref{eq:prodzg}) is written as a product of elements of $UVP_n$, which we will now project onto $F_{i_1,j_1}$. Let $\alpha(\lambda_{i_1,j_1}, \lambda_{j_1,i_1})=\pi_{i_1,j_1}(z)$ et $\beta(\lambda_{i_1,j_1}, \lambda_{j_1,i_1})=\pi_{i_1,j_1}(g)$. 
Using~(\ref{eq:conjuvpn}) and the fact that $\pi_{UVP}(\rho)=(i_1,j_1)\cdots (i_r,j_r)$, by~(\ref{eq:prodzg}), we have: 
\begin{equation*}
\pi_{i_1,j_1}((\eta(\widehat{\beta}))^{-1}\ldotp v)=(\alpha(\lambda_{i_1,j_1}, \lambda_{j_1,i_1}))^{-1} \ldotp \alpha(\lambda_{j_1,i_1}, \lambda_{j_1,i_1}) \ldotp \lambda_{j_1,i_1} \ldotp \beta(\lambda_{j_1,i_1},\lambda_{i_1,j_1}) \ldotp (\beta(\lambda_{i_1,j_1}, \lambda_{j_1,i_1}))^{-1},
\end{equation*}
from which it follows that $\varepsilon_{i_1,j_1}(\eta(\widehat{\beta}))^{-1}\ldotp v)=1$. We conclude from the previous paragraph that $\eta(\widehat{\beta}))^{-1}\ldotp v \notin \Gamma$, and thus $v\notin \operatorname{\text{Im}}(\eta)$, which yields a contradiction. Therefore $\operatorname{\text{Im}}(\eta)$ contains no elements of order $2$, and we deduce from Proposition~\ref{prop:cris} that $B_n/[P_n,P_n]$ has no elements of even order.\qedhere
\end{enumerate}
\end{proof}

\begin{prop}\label{prop:torsuvbn}
If $n\geq 3$ then any torsion element of $UVB_n$ belongs to the normal closure of $\operatorname{\text{Im}}(\iota)$ in $UVB_n$.
\end{prop}

\begin{proof}
From Definition~\ref{D:Unrestricted}, one may check that the quotient of $UVB_n$ by the normal closure of $\operatorname{\text{Im}}(\iota)$ in $UVB_n$ is isomorphic to $\Z$. If $g$ is a torsion element of $UVB_n$, its image in this quotient is thus trivial, in other words $g$ belongs to the normal closure of $\operatorname{\text{Im}}(\iota)$ in $UVB_n$.
\end{proof}

One may show in the same manner that the statement of Proposition~\ref{prop:torsuvbn} also holds for $VB_n$ and $WB_n$.
In the case of $UVB_n$, we may strengthen the results of Theorem~\ref{thm:ord2}(\ref{it:ord2a}) and Proposition~\ref{prop:torsuvbn}.

\begin{thm}\label{thm:thm_of_torsion}
Let $n\geq 2$, and let $w$ be a torsion element of $UVB_n$ of order $r$. Then there exists $s_w \in S_n$ of order $r$ such that $w$ is conjugate to $\iota(s_w)$ by an element of $UVP_n$.
\end{thm}

Before proving Theorem~\ref{thm:thm_of_torsion}, we define some notation, and we make study the action of the symmetric group $S_n$ on the group $UVP_n$ that will used in the proof. If $\tau\in S_n$, let $\operatorname{\text{Supp}}(\tau)$ denote its support. Let $s\in S_n$, let $o(s)$ denote  the order of $s$ and let $G_s$ denote the cyclic subgroup of $S_n$ of order $o(s)$ generated by the element $s$. By Theorem~\ref{thm:tss2}, the permutation $s$ acts on $UVP_n=\langle\lambda_{i,j}\ |\ 1\leq i\neq j\leq n\rangle$ by permuting the indices of the elements $\lambda_{i,j}$. To simplify the notation, if $g \in UVP_n$ and $s\in S_n$, we set $s(g)= \iota(s) g \iota(s)^{-1}$, and we shall identify $s$ with its image $\iota(s)$ in $UVB_n$. Let $s=s_1 \cdots s_{m(s)}$ be the cycle decomposition of $s$, where $1\leq m(s)\leq n$, and the subsets $\operatorname{\text{Supp}}(s_1),\ldots, \operatorname{\text{Supp}}(s_{m(s)})$ form a partition of the set $\{1,\ldots, n \}$ (in particular, some of the $\operatorname{\text{Supp}}(s_j)$, where $1\leq j\leq m(s)$, may be singletons). 

Let $I=\{ (i, j) \mid 1\le i\neq j \le n\}$. The action of $G_s$ on $I$ gives rise to a partition of $I$ in $n(s)$ disjoint orbits, where $n(s)\in \N$, and the cardinality of each such orbit is either the order of one cycle or is the least common multiple of the orders of two cycles of $s$. If $(i, j)\in I$, let $\mathcal{O}(i,j)$ denote its orbit, and let $\lvert \mathcal{O}(i,j)\rvert$ denote the cardinality of $\mathcal{O}(i,j)$. Clearly, $\lvert \mathcal{O}(j,i)\rvert= \lvert \mathcal{O}(i,j)\rvert$ for all $(i,j)\in I$. Further, if $\mathcal{O}(j,i)=\mathcal{O}(i,j)$ then there exist $1\leq q\leq m(s)$ and $p\in \N$ such that $i,j\in \operatorname{\text{Supp}}(s_q)$, $s_q^p(i)=j$ and $s_q^p(j)=i$, from which it follows that $\lvert \mathcal{O}(i,j)\rvert= o(s_q)$ is even, and if $(k,l)\in I$ then $(k,l)\in \mathcal{O}(i,j)$ if and only if $(l,k)\in \mathcal{O}(i,j)$. This gives rise to a partition $\sqcup_{k=1}^{n(s)} \mathcal{O}_k$  of $I$, where $n(s)\in \N$, that dominates the partition given by the action of $G_s$ on $I$, and for $k=1,\ldots, n(s)$, $\mathcal{O}_k$ is defined by the property that if $(i,j)\in I$, then $(i,j)\in \mathcal{O}_k$ if and only if $\mathcal{O}_k=\mathcal{O}(i,j) \cup \mathcal{O}(j,i)$. Note that $n(s)$ is the number of orbits of the action of $G_s$ on the set $\{ \{i, j\} \mid 1\le i\neq j \le n\}$ of \emph{unordered} pairs of elements of $\{ 1,\ldots,n\}$. 
Since $\lvert \mathcal{O}(j,i)\rvert= \lvert \mathcal{O}(i,j)\rvert$, and $\mathcal{O}(i,j)$ is even if $\mathcal{O}(j,i)=\mathcal{O}(i,j)$, it follows that for all $1\leq k\leq n(s)$, the cardinality $\lvert\mathcal{O}_k\rvert$ of $\mathcal{O}_k$ is even, so $\lvert\mathcal{O}_k\rvert=2n_k$ for some $n_k\in \N$. 
Let $F_{\mathcal{O}_k} =\bigoplus_{(i, j) \in \mathcal{O}_k} F_{i,j}$, where as in Section~\ref{sec:uvb}, we identify $F_{i,j}$ with $F_{j,i}$. Then:
\begin{equation}\label{eq:decompFOk}
UVP_n= \bigoplus_{1\leq k\leq n(s)} F_{\mathcal{O}_k}.
\end{equation}

 
 
The following lemma is folklore.
 
\begin{lem}\label{prep:lemma1}
Let $F_2=F_2 (x, y)$ be the free group of rank $2$ freely generated by $\{ x, y\}$, and let $\alpha\colon\thinspace F_2  \to  F_2$ be the involutive automorphism defined by $\alpha(x)=y$ and $\alpha(y)=x$. Let $w \in F_2$ be such that $w \alpha(w)=1$. Then there exists $u \in F_2$ such that $w=u \alpha(u^{-1})$.
\end{lem}

\begin{proof}
Let $w \in F_2$ be such that $w \alpha(w)=1$. If $w=1$ then we may take $u=1$. So suppose that $w\neq 1$. Since $w\alpha(w)=1$, $w$ cannot be written in reduced form as a word starting and ending with a non-zero power of the same generator. By replacing $w$ by $\alpha(w)$ if necessary, we may thus suppose that $w=x^{\varepsilon_1}y^{\varepsilon_2}\cdots y^{\varepsilon_{2t-1}}y^{\varepsilon_{2t}}\neq1$, where $\varepsilon_1,\varepsilon_2,\dots,\varepsilon_{2t} \in \Z \setminus\{0\}$, and thus:
\begin{equation*}
1=w \alpha(w)=\underbrace{\big(x^{\varepsilon_1}y^{\varepsilon_2}\cdots x^{\varepsilon_{2t-1}}y^{\varepsilon_{2t}}\big)}_{\neq1}\cdot\underbrace{\big(y^{\varepsilon_1}x^{\varepsilon_2}\cdots y^{\varepsilon_{2t-1}}x^{\varepsilon_{2t}}\big)}_{\neq1}\in F_2.
\end{equation*}
It follows that $ \varepsilon_{2t-j} = - \varepsilon_{j+1}$  for $j=0,\ldots,2t-1$, from which we conclude that $w= u \alpha(u^{-1})$, where $u= x^{\varepsilon_1}y^{\varepsilon_2}\cdots y^{\varepsilon_{t}}$.
\end{proof}

In what follows, we shall identify $S_n$ with its image in $UVB_n$ by $\iota$. In particular, if $s\in S_n$, then $\iota(s)$ shall be denoted simply by $s$.
  
\begin{lem}\label{prep:lemma2}
Let $n\geq 2$, let $s\in S_n$, let $\gamma_1 \in F_{\mathcal{O}_l}$, where $1\leq l\leq n(s)$, and let $w_1=\gamma_1 s \in UVB_n$. If $w_1$ is of finite order then it is conjugate to $s$ by an element of $F_{\mathcal{O}_l}$.
\end{lem}

\begin{proof}
Let $s\in S_n$, $\gamma_1 \in F_{\mathcal{O}_l}$, where $1\leq l\leq n(s)$, and $w_1\in UVB_n$ be as in the statement. By reordering the orbits $\mathcal{O}_{1},\ldots, \mathcal{O}_{n(s)}$ if necessary, we may suppose that $l=1$, so $\gamma_1 \in F_{\mathcal{O}_1}$. The result is clear if $o(w_1)=1$. If $o(w_1)=2$, then by Theorem~\ref{thm:ord2}(\ref{it:ord2a}), there exists $g\in UVP_n$ such that $w_1=gsg^{-1}$. With respect to the decomposition~(\ref{eq:decompFOk}), let $g=\delta_1 \cdots \delta_{n(s)}$, where for $k=1,\ldots, n(s)$, $\delta_k \in F_{\mathcal{O}_k}$. Since $gsg^{-1}=\gamma_1 s$, we have:
\begin{equation}\label{eq:gamma1a}
\gamma_1 =\delta_1 \cdots \delta_{n(s)} \ldotp s \, (\delta_1 \cdots \delta_{n(s)})^{-1}s^{-1} = \delta_1 \cdots \delta_{n(s)} \ldotp s \delta_{n(s)}^{-1}s^{-1} \cdots s \delta_{1}^{-1}s^{-1}\in F_{\mathcal{O}_1}. 
\end{equation}
Now for all $k=1,\ldots, n(s)$, $F_{\mathcal{O}_k}$ is invariant under conjugation by $s$, and identifying the components of~(\ref{eq:gamma1a}) with respect to~(\ref{eq:decompFOk}), we see that $\delta_k \ldotp s \delta_k^{-1}s^{-1}=1$ for all $2\leq k\leq n(s)$, or in other words, $\delta_k$ and $s$ commute. It follows from~(\ref{eq:gamma1a}) that $\gamma_1=\delta_1 \ldotp s \delta_1^{-1}s^{-1}$, and thus $w_1=\gamma_1 s=\delta_1 s \delta_1^{-1}$, where $\delta_1\in F_{\mathcal{O}_1}$, as required. 


So suppose that $o(w_1)\geq 3$. 
For all $m\in \N$, we have:
\begin{equation}\label{eq:w1m}
w_1^m= (\gamma_1 s)^m= \gamma_1s\gamma_1s^{-1}s^2\gamma_1s^{-2}
\cdots s^{m-1}\gamma_1s^{-(m-1)}s^m=\gamma_1 s(\gamma_1) \cdots s^{m-1}(\gamma_1) s^m.
\end{equation}
Since $UVP_n$ is torsion free, it follows from~(\ref{eq:w1m}) that $o(w_1)=o(s)$. As we mentioned above, $F_{\mathcal{O}_1}$ is invariant under conjugation by $s$, so for all $j=0,1,\ldots, m-1$, the term $s^j(\gamma_1)= s^j \gamma_1 s^{-j}$ of~(\ref{eq:w1m}) belongs to $F_{\mathcal{O}_1}$.
Let $(i_0,j_0) \in \mathcal{O}_1$, where $1\leq i_0<j_0\leq n$. 
If $\mathcal{O}(i_0,j_0)\neq \mathcal{O}(j_0,i_0)$ (resp.\ $\mathcal{O}(i_0,j_0)= \mathcal{O}(j_0,i_0)$), let $\varepsilon=1$ (resp.\ $\varepsilon=2$). Then 
$\lvert \mathcal{O}(i_0,j_0)\rvert=\varepsilon n_1$, and $\lvert \mathcal{O}_1\rvert=2n_1$, $F_{\mathcal{O}_1}=\bigoplus_{(i, j) \in \mathcal{O}_1} F_{i,j}=\bigoplus_{q=0}^{n_1-1} F_{s^q(i_0),s^q(j_0)}$, and for $j=0, \ldots, n_1-1$, there exists $u_j \in F_{s^j(i_0),s^j(j_0)}$ such that $\gamma_1=u_0\cdots u_{n_1-1}$. Hence $u_0, \ldots, u_{n_1-1}$ belong to distinct free factors of $UVP_n$, so commute pairwise. Further:
\begin{equation}\label{eq:skuj}
\text{$s^k(u_j)\in F_{s^{j+k}(i_0),s^{j+k}(j_0)}$ for all $k\in \Z$ and $j=0, \ldots, n_1-1$,}
\end{equation}
where the superscript $j+k$ is taken modulo $n_1$. In particular, if $v \in F_{s^q(i_0),s^q(j_0)}$ for some $q=0,\ldots, n_1-1$, then $s^{n_1}(v) = v$ (resp.\ $s^{n_1}(v) = \alpha(v)$, where $\alpha\colon\thinspace F_{s^q(i_0),s^q(j_0)} \to F_{s^q(i_0),s^q(j_0)}$ is the automorphism that exchanges $\lambda_{s^q(i_0),s^q(j_0)} $ and $\lambda_{s^q(j_0),s^q(i_0)}$, as in Lemma~\ref{prep:lemma1}), and thus $s^{\varepsilon n_1}(v) = v$. It follows from this, the fact that $o(w_1)=o(s)$ and~(\ref{eq:w1m}) that:
\begin{align*}
1&=w_1^{\varepsilon n_1 o(s)}=(w_1^{\varepsilon n_1})^{o(s)}=(\gamma_1 s(\gamma_1) \cdots s^{\varepsilon n_1-1}(\gamma_1) s^{\varepsilon n_1})^{o(s)}=(\gamma_1 s(\gamma_1) \cdots s^{\varepsilon n_1-1}(\gamma_1))^{o(s)} s^{\varepsilon n_1 o(s)}\\
&=(\gamma_1 s(\gamma_1) \cdots s^{\varepsilon n_1-1}(\gamma_1))^{o(s)}.
\end{align*}
Since $UVP_n$ is torsion free, we conclude that $\gamma_1 s(\gamma_1)\cdots s^{\varepsilon n_1-1}(\gamma_1)=1$. Hence:
\begin{equation}\label{eq:prodsu}
u_0\cdots u_{n_1-1}  s\big(u_0\cdots u_{n_1-1}\big) \cdots s^{\varepsilon n_1-1}\big(u_0\cdots u_{n_1-1}\big)=1.
\end{equation}
Applying~(\ref{eq:skuj}), and projecting~(\ref{eq:prodsu}) into $F_{1,2}$, we see that:
\begin{equation}\label{eq:prodsin1}
1 =\prod_{i=0}^{\varepsilon-1} s^{in_1}(u_0) s^{in_1+1}\big( u_{n_1-1}\big)\cdots s^{(i+1)n_1-1}\big(u_1 \big)= \prod_{i=0}^{\varepsilon-1} s^{in_1}\big( u_0 s\big(u_{n_1-1}\big)\cdots s^{n_1-1}\big(u_1 \big)\big).
\end{equation}
We deduce from~(\ref{eq:prodsin1}) (resp.\ from~(\ref{eq:prodsin1}) and Lemma~\ref{prep:lemma1}) that $u_{0}=\prod_{j=1}^{n_1-1}s^{n_1-j}\big( u_j^{-1} \big)$ (resp.\ that there exists $\widetilde{u} \in F_{1,2}$ such that $u_{0}=\widetilde{u} s^{n_1}(\widetilde{u}^{-1})  \big(\prod_{j=1}^{n_1-1}s^{n_1-j}\big( u_j^{-1} \big)\big)$). Using the fact that $u_1,\ldots, u_{n_1-1}$ commute pairwise, it follows that:
\begin{equation}\label{eq:gamma1}
\gamma_{1}= u_{0}u_{1}\cdots u_{n_1-1}= \beta s^{n_1-1}\big(u_1^{-1}\big)\cdots s\big(u_{n_1-1}^{-1} \big) \cdot u_{n_1-1} \cdots u_1,
\end{equation}
where $\beta=1$ (resp.\ $\beta=\widetilde{u} s^{n_1}(\widetilde{u}^{-1})$). In what follows, if $a,b\in UVB_n$, we shall write $a\sim b$ if $a$ and $b$ are conjugate by an element of $UVP_n$. Let us show by induction that:
\begin{equation}\label{eq:gamma1s}
\gamma_1\cdot s \sim \beta s^{n_1-1}\big(u_1^{-1}\big)\cdots s^{t+1}\big(u_{n_1-(t+1)}^{-1} \big) s^t\big(u_{n_1-t}^{-1} \big) \cdot u_{n_1-t} \cdots u_1 \cdot s
\end{equation}
for all $t=1,\ldots, n_1$. If $t=1$ then the result follows directly from~(\ref{eq:gamma1}). So suppose that~(\ref{eq:gamma1s}) holds for some $1\leq t\leq n_1-1$. Let $M_0= \prod_{c=1}^{n_1-1} s^{c}(\widetilde{u})$, and for $1\leq t\leq n_1-1$, let $M_t=\prod_{c= 1}^{t-1}s^{t-c}\big(u_{n_1-t} \big)$. Note that $M_1=1$, and that $M_0=1$ if $\mathcal{O}(i_0,j_0)\neq \mathcal{O}(j_0,i_0)$. By~(\ref{eq:skuj}), for $c=1,\ldots,t-1$, $s^{t-c}\big(u_{n_1-t} \big)\in F_{s^{n_1-c}(i_0),s^{n_1-c}(j_0)}$, so for $1\leq t\leq n_1-1$, the factors of $M_t$ commute pairwise, and $M_t\in \bigoplus_{c=n_1-t+1}^{n_1-1} F_{s^{c}(i_0),s^{c}(j_0)}$. Thus:
\begin{equation}\label{eq:mtst}
M_t^{-1}s^t\big(u_{n_1-t}^{-1} \big)=\left(\prod_{c= 1}^{t-1}s^{t-c}\big(u_{n_1-t} \big)\right)^{-1} s^t\big(u_{n_1-t}^{-1} \big)= \left(\prod_{c=1}^{t-1}s^{c}\big(u_{n_1-t}^{-1} \big)\right) s^t\big(u_{n_1-t}^{-1} \big)=\prod_{c=1}^{t}s^{c}\big(u_{n_1-t}^{-1} \big).
\end{equation}
and since $s^t\big(u_{n_1-t}^{-1}\big)\in F_{i_0,j_0}$ and $u_{n_1-t} \cdots u_1\in \bigoplus_{i=1}^{n_1-t} F_{s^{i}(i_0),s^{i}(j_0)}$, we see that $M_t^{-1}s^t\big(u_{n_1-t}^{-1} \big)$ commutes with $u_{n_1-t} \cdots u_1$. So by~(\ref{eq:gamma1s}) and~(\ref{eq:mtst}), we have: 
\begin{align}
\gamma_1\cdot s &\sim \beta s^{n_1-1}\big(u_1^{-1}\big)\cdots s^{t+1}\big(u_{n_1-(t+1)}^{-1} \big) M_t M_t^{-1} s^t\big(u_{n_1-t}^{-1} \big) \cdot u_{n_1-t} \cdots u_1 \cdot s\notag\\
&= \beta s^{n_1-1}\big(u_1^{-1}\big)\cdots s^{t+1}\big(u_{n_1-(t+1)}^{-1} \big) M_t \cdot u_{n_1-t} \cdots u_1 \left(\prod_{c=1}^{t}s^{c}\big(u_{n_1-t}^{-1} \big)\right) \cdot s\notag\\
&= \beta s^{n_1-1}\big(u_1^{-1}\big)\cdots s^{t+1}\big(u_{n_1-(t+1)}^{-1} \big) \left(\prod_{c=1}^{t} s^{t-c}\big(u_{n_1-t} \big)\right) u_{n_1-(t+1)} \cdots u_1 \cdot s \cdot \prod_{c=0}^{t-1}s^{c}\big(u_{n_1-t}^{-1} \big),\label{eq:gamma1scom}
\end{align}
where we have used the fact that $M_t \cdot u_{n_1-t}=\prod_{c=1}^{t} s^{t-c}\big(u_{n_1-t} \big)$. Now by~(\ref{eq:skuj}), $\prod_{c=1}^{t} s^{t-c}\big(u_{n_1-t} \big)\in \bigoplus_{c=n_1-t}^{n_1-1} F_{s^{c}(i_0),s^{c}(j_0)}$, and since $\beta s^{n_1-1}\big(u_1^{-1}\big)\cdots s^{t+1}\big(u_{n_1-(t+1)}^{-1} \big)\in F_{i_0,j_0}$, we see that these two terms commute, and it follows from~(\ref{eq:gamma1scom}) that:
\begin{equation}\label{eq:gamma1scom2}
\gamma_1\cdot s \sim \left(\prod_{c=1}^{t} s^{t-c}\big(u_{n_1-t} \big)\right) \beta s^{n_1-1}\big(u_1^{-1}\big)\cdots s^{t+1}\big(u_{n_1-(t+1)}^{-1} \big)  u_{n_1-(t+1)} \cdots u_1 \cdot s \cdot \prod_{c=0}^{t-1}s^{c}\big(u_{n_1-t}^{-1} \big).
\end{equation}
One may check that $\prod_{c=1}^{t} s^{t-c}\big(u_{n_1-t} \big)$ is the inverse of $\prod_{c=0}^{t-1}s^{c}\big(u_{n_1-t}^{-1} \big)$, and using~(\ref{eq:gamma1scom2}), we conclude that~(\ref{eq:gamma1s}) holds for $t+1$. Taking $t=n_1$ in~(\ref{eq:gamma1s}), we obtain $\gamma_1\cdot s \sim \beta \cdot s$. If $\mathcal{O}(i_0,j_0)\neq \mathcal{O}(j_0,i_0)$ then  $\beta=1$, and the statement of the lemma holds in this case. So suppose that $\mathcal{O}(i_0,j_0)= \mathcal{O}(j_0,i_0)$. Then $\widetilde{u}\in F_{i_0,j_0}$, and for $c=1,\ldots, n_1-1$, $s^c(\widetilde{u})\in F_{s^c(i_0),s^c(j_0)}$. Thus the factors of $\widetilde{u} M_0=\prod_{c=0}^{n_1-1} s^{c}(\widetilde{u})$ commute pairwise, and hence $M_0^{-1} s^{n_1}(\widetilde{u}^{-1})= \prod_{c=1}^{n_1} s^{c}(\widetilde{u}^{-1})$. It follows that:
\begin{align*}
\gamma_1\cdot s \sim \beta \cdot s &=\widetilde{u} s^{n_1}(\widetilde{u}^{-1})\cdot s =  \widetilde{u} M_0 M_0^{-1} s^{n_1}(\widetilde{u}^{-1})\cdot s=\left(\prod_{c=0}^{n_1-1} s^{c}(\widetilde{u})\right) \left(\prod_{c=1}^{n_1} s^{c}(\widetilde{u}^{-1})\right) \cdot s\\
&=\left(\prod_{c=0}^{n_1-1} s^{c}(\widetilde{u})\right) \cdot s \cdot \left(\prod_{c=0}^{n_1-1} s^{c}(\widetilde{u}^{-1})\right).
\end{align*}
Now $\prod_{c=0}^{n_1-1} s^{c}(\widetilde{u})$ may be seen to be the inverse of $\prod_{c=0}^{n_1-1} s^{c}(\widetilde{u}^{-1})$, and therefore $\gamma_1\cdot s \sim s$ also in this case, which completes the proof of the lemma.
\end{proof}

This enables us to prove Theorem~\ref{thm:thm_of_torsion}.

\begin{proof}[Proof of Theorem~\ref{thm:thm_of_torsion}]
Let $w\in UVB_n$ be a torsion element of order $r$. By Theorems~\ref{thm:tss} and~\ref{thm:tss2} there exist (unique) $u\in UVP_n$ and $s\in S_n$ such that $w=us$. As in~(\ref{eq:w1m}), we see that $u\cdot s(u)\cdots s^{r-1}(u)=1$ in $UVP_n$ and $s^r=1$. Further, from~(\ref{eq:decompFOk}), $u=\gamma_1\cdots\gamma_{n(s)}$, where for $l=1,\ldots, n(s)$, $\gamma_l \in F_{\mathcal{O}_l}$.
Let $l\in \{1,\ldots, n(s)\}$. Since $F_{\mathcal{O}_l}$ is invariant under conjugation by $s$, it follows from~(\ref{eq:decompFOk}) that $\gamma_l s(\gamma_l)\cdots s^{r-1}(\gamma_l)=1$ in $F_{\mathcal{O}_l}$. Since $\gamma_l s(\gamma_l)\cdots s^{r-1}(\gamma_l)$ may also be written as $(\gamma_l s)^r$ using~(\ref{eq:w1m}), we conclude from Lemma~\ref{prep:lemma2} that $\gamma_l s$ is conjugate to $s$ by an element of $F_{\mathcal{O}_l}$, or in other words, there exists $\Lambda_l\in F_{\mathcal{O}_l}$ such that $\gamma_l s=\Lambda_{l}\cdot s\cdot \Lambda_{l}^{-1}$.
Let us prove by reverse induction that for all $m=1,\ldots, n(s)+1$:
\begin{equation}\label{eq:Lambdam}
w= \Lambda_{n(s)}\Lambda_{n(s)-1} \cdots \Lambda_{m} \gamma_1 \cdots \gamma_{m-1}\cdot s \cdot \Lambda_{m}^{-1} \cdots \Lambda_{n(s)-1}^{-1}\Lambda_{n(s)}^{-1}.
\end{equation}
If $m=n(s)+1$ then~(\ref{eq:Lambdam}) follows directly from the fact that $w=us=\gamma_1\cdots \gamma_{n(s)}\cdot s$. So suppose that~(\ref{eq:Lambdam}) holds for some $m\in \{2,\ldots, n(s)+1 \}$. Then using~(\ref{eq:decompFOk}) and the facts that $\gamma_1 \cdots \gamma_{m-2}\in \bigoplus_{i=1}^{m-2} F_{\mathcal{O}_i}$ and $\gamma_{m-1}\in F_{\mathcal{O}_m-1}$, we obtain:
\begin{align*}
w &= \Lambda_{n(s)}\Lambda_{n(s)-1} \cdots \Lambda_{m} \gamma_1 \cdots \gamma_{m-2} (\gamma_{m-1} \cdot s) \cdot \Lambda_{m}^{-1} \cdots \Lambda_{n(s)-1}^{-1}\Lambda_{n(s)}^{-1}\\
&= \Lambda_{n(s)}\Lambda_{n(s)-1} \cdots \Lambda_{m} \gamma_1 \cdots \gamma_{m-2} \Lambda_{m-1} \cdot s \cdot \Lambda_{m-1}^{-1}  \Lambda_{m}^{-1} \cdots \Lambda_{n(s)-1}^{-1}\Lambda_{n(s)}^{-1}\\
&= \Lambda_{n(s)}\Lambda_{n(s)-1} \cdots \Lambda_{m} \Lambda_{m-1} \gamma_1 \cdots \gamma_{m-2} \cdot s \cdot \Lambda_{m-1}^{-1}  \Lambda_{m}^{-1} \cdots \Lambda_{n(s)-1}^{-1}\Lambda_{n(s)}^{-1},
\end{align*}
which proves~(\ref{eq:Lambdam}) for $m-1$. Taking $m=1$, we see that $w=\Lambda \cdot s \cdot \Lambda^{-1}$, where $\Lambda=\Lambda_{n(s)}\cdots \Lambda_{1}\in UVP_n$, and this completes the proof of the theorem.
\end{proof}

\begin{rem}
It is an open question whether a result similar to that of Theorem~\ref{thm:thm_of_torsion} holds for $VB_n$ and for $WB_n$.
\end{rem}




\begin{acknow}
The first and third authors are supported by the French project `AlMaRe' (ANR-19-CE40-0001-01). 
\end{acknow}


\begin{thebibliography}{MKS}

\bibitem{ABMW} B.~Audoux, P.~Bellingeri, J.-B.~Meilhan and E.~Wagner,
Extensions of Some Classical Local Moves on Knot Diagrams,
Michigan Math.~J.\ 67 (2018), 647--672.

\bibitem{Bar-0}
V.~G.~Bardakov, The virtual and universal braids, Fund.\ Math.\ 184 (2004), 1--18.


\bibitem{BBD}
V.~G.~Bardakov, P.~Bellingeri and C.~Damiani,
Unrestricted virtual braids, fused links and other quotients of virtual braid groups, J.~Knot Theory Ramifications 24 (2015), 23 pp.

	
\bibitem{BM} V.~Beck and I.~Marin, Torsion subgroups of quasi-Abelianized braid groups. J.~Algebra 558 (2020), 3--23.

\bibitem{BP}
P.~Bellingeri and L.~Paris, Virtual braids and permutations, Ann.\ Inst.\ Fourier 70 (2020), 1341--1362.

\bibitem{BPT}
P.~Bellingeri, L.~Paris and A.-L.~Thiel, Virtual Artin groups, preprint, \url{arXiv:2110.14293}.

\bibitem{BS} P.~Bellingeri and A.~Souli\'e, A note on representations of welded braid groups, J.~Knot Theory Ramifications 29 (2020), 21 pp.

 
\bibitem{BH}
T.~Brendle and A.~Hatcher,
Configuration spaces of rings and wickets, Commentarii Math.\ Helv.\ 88 (2013), 131--162. 

\bibitem{CO}
P.~C.~Cerqueira Dos Santos and O.~Ocampo, Virtual braid groups, virtual twin groups and crystallographic groups, preprint, \url{arXiv:2110.02392}.

\bibitem{Cis}
B.~A.~Cisneros De La Cruz, Virtual braids from a topological viewpoint, J.~Knot Theory Ramifications 24 (2015), 36 pp.

\bibitem{D}
C.~Damiani, 
A journey through loop braid groups, Expositiones Mathematicae 35 (2017), 252--285.
  
\bibitem{Dek}
K.~Dekimpe, Almost Bieberbach Groups: Affine and Polynomial Structures, Springer LNM 1639, Berlin (1996).

\bibitem{Di}
F.~Digne, Présentation des groupes de tresses purs et de certaines de leurs extensions, preprint, \url{arXiv:1511.08731}.
 

\bibitem{FRR}
R.~Fenn, R.~Rim\'{a}nyi and C.~Rourke, The braid--permutation group,
Topology 36 No.1 (1997), 123--135.

\bibitem{G}
R.~Gaudreau, Classification of virtual string links up to cobordism,  Ars Math.\ Contemp.\ 19 (2020), 37--49.

\bibitem{GGO17} D.~L.~Gon{\c{c}}alves, J.~Guaschi and O.~Ocampo,
A quotient of the Artin braid groups related to crystallographic groups, J.~Algebra 474 (2017), 393--423.

\bibitem{GGO19} D.~L.~Gon\c{c}alves, J.~Guaschi and O.~Ocampo, Almost-crystallographic groups as quotients of Artin braid groups, J.~Algebra 524 (2019), 160--186.

\bibitem{GGO20} D.~L.~Gon\c{c}alves, J.~Guaschi and O.~Ocampo, Embeddings of finite groups in $B_{n}/\Gamma_{k}(P_{n})$ for $k=2,3$, Ann.~Inst.~Fourier 70 (2020), 2005--2025.

\bibitem{GGOP} D.~L.~Gonçalves, J.~Guaschi, O.~Ocampo and C.~de Miranda e Pereiro, Crystallographic groups and flat manifolds from surface braid groups, Top.\ Appl.\ 293 (2021), 107560.

\bibitem{KMRW} Z.~K{\'a}d{\'a}r, P.~Martin, E.~Rowell and Z.~Wang,  
Local representations of the loop braid group, Glasgow Math.~J.\
59 (2017), 359--378.
 
\bibitem{Kam} S.~Kamada, Braid presentation of virtual knots and welded knots, Osaka J.~Math.\ 44 no.~2 (2007), 441--458.








\bibitem{KaL}
L.~H.~Kauffman and S.~Lambropoulou, Virtual braids and the L-move, J.~Knot Theory Ramifications 15 (2006), 1--39.

\bibitem{M09}
I.~Marin  Reflection groups acting on their hyperplanes, J.~Algebra 322 (2009), 2848--2860. 

\bibitem{M16} I.~Marin, Crystallographic groups and flat manifolds from complex reflection groups, Geom.\ Dedicata 182 (2016), 233--247.

\bibitem{N} T.~Nasybullov, The classification of fused links, J.~Knot Theory Ramifications 25 (2016), no.~21.

\bibitem{Su}
A.~Suciu, The Pure Braid Groups and Their Relatives, Perspectives in Lie Theory pp.~403--426, Springer INdAM Ser., 19, Springer, Cham, 2017.
 
\bibitem{Ti}
J.~Tits, Normalisateurs de tores I~: Groupes de Coxeter étendus, J.~Algebra 4 (1966), 96--116.

\end{thebibliography}
\end{document}